\documentclass[10pt]{article}
\usepackage{amsmath,amssymb,amsthm}
\usepackage{epsfig}
\usepackage{color}

\title{Signed edge domination numbers of complete tripartite graphs: Part 2}
\author {
\begin{tabular}{c}
Abdollah Khodkar\\
Department of Mathematics\\
University of West Georgia\\
Carrollton, GA 30118\\
{\tt akhodkar@westga.edu}
\end{tabular}
}

\date{}

\newtheorem{prelem}{{\bf Theorem}}

 \newtheorem{theorem}{Theorem}
\newtheorem{corollary}[theorem]{Corollary}
\newtheorem{lemma}[theorem]{Lemma}

\theoremstyle{definition}

\theoremstyle{remark}

\begin{document}
\maketitle

\begin{abstract}
\noindent The closed neighborhood $N_G[e]$ of an edge $e$ in a graph $G$ is
the set consisting of $e$ and of all edges having an end-vertex in
common with $e$.  Let $f$ be a function on $E(G)$, the edge set of
$G$, into the set $\{-1,1\}$. If
$\sum_{x\in{N[e]}}f(x)\geq 1$ for each edge $e \in E(G)$, then $f$
is called a signed edge dominating function of $G$.
The signed edge domination number of $G$ is
the minimum weight of a signed edge dominating function of $G$.
In this paper, we find the signed edge
domination number of the complete tripartite
graph $K_{m,n,p}$, where $1\leq m\leq n$ and $p\geq m+n$.
This completes the search for the signed edge domination numbers of the complete tripartite graphs.
\vspace{3mm}
\end{abstract}

\section{Introduction}\label{SEC1}

Let $G$ be a simple non-empty graph with vertex set $V(G)$ and edge
set $E(G)$.  We use \cite{W} for terminology and notation not
defined here. Two edges $e_1$, $e_2$ of $G$ are called {\em
adjacent} if they are distinct and have a common end-vertex.  The
{\em open neighborhood} $N_G(e)$ of an edge $e\in E(G)$ is the set
of all edges adjacent to $e$.  Its {\em closed neighborhood} is
$N_G[e] = N_G(e)\cup\{e\}$. For a function $f : E(G) \rightarrow
\{-1,1\}$ and a subset $S$ of $E(G)$ we define $f(S) = \sum_{x\in
S}f(x)$.  If $S = N_{G}[e]$ for some $e\in E$, then we denote $f(S)$
by $f[e]$.  The {\em weight} of vertex $v \in V(G)$ is defined by
$f(v) = \sum_{e\in E(v)}f(e)$, where $E(v)$ is the set of all edges at
vertex $v$. A function $f : E(G)\rightarrow \{-1,1\}$ is called a {\em signed edge dominating
function} (SEDF) of $G$ if $f[e] \geq 1$ for each edge $e \in
E(G)$. The SEDF of a graph was first defined in \cite{Xu1}.
The {\em weight} of $f$, denoted $w(f)$, is defined to be
$w(f)=\sum_{e\in E(G)}f(e)$. {\em The signed edge domination
number} (SEDN) $\gamma'_{s}(G)$ is defined as $\gamma'_{s}(G)$ =
min\{$w(f) \mid f$ is an SEDF of $G$\}. An SEDF $f$ is called a $\gamma'_{s}(G)$-function if
$\omega(f)=\gamma'_{s}(G)$.
In \cite{Xu1} it was conjectured that $\gamma_s'(G)\leq |V(G)|-1$ for every graph $G$ of order at least 2.

The signed edge domination numbers of the complete graph $K_n$ and
the complete bipartite graph $K_{m,n}$ were
determined in \cite{Xu2} and \cite{ABHS}, respectively.
In \cite{KG}, the signed edge domination number of
$K_{m,n,p}$ was calculated when $1\leq m\leq n\leq p\leq m+n$.
For completeness, we state the main theorem of \cite{KG}.

\begin{theorem}\label{main1}
{\rm
Let $m,n$ and $p$ be positive integers and $m\leq n\leq p\leq m+n$.
Let $(m,n,p)\not\in \{(1,1,1),$ $(2,3,5)\}$.
\begin{itemize}
\item[{\bf A.}] Let $m$, $n$ and $p$ be even.
   \begin{enumerate}
     \item If $m+n+p\equiv 0$ (mod 4), then $\gamma_s'(K_{m,n,p})=(m+n+p)/2$.

    \item If $m+n+p\equiv 2$ (mod 4), then $\gamma_s'(K_{m,n,p})=(m+n+p+2)/2$.
   \end{enumerate}

\item[{\bf B.}] Let $m,n$ and $p$ be odd.
    \begin{enumerate}
      \item If $m+n+p\equiv 1$ (mod 4), then $\gamma_s'(K_{m,n,p})=(m+n+p+1)/2$.

     \item If $m+n+p\equiv 3$ (mod 4), then $\gamma_s'(K_{m,n,p})=(m+n+p+3)/2$.
    \end{enumerate}

\item[{\bf C.}] Let $m,n$ be odd and $p$ be even or $m,n$ be even and $p$ be odd.
    \begin{enumerate}
      \item If $m+n\equiv 0$ (mod 4), then $\gamma_s'(K_{m,n,p})=(m+n)/2+p+1$.
      \item If $m+n\equiv 2$ (mod 4), then $\gamma_s'(K_{m,n,p})=(m+n)/2+p$.
    \end{enumerate}

\item[{\bf D.}] Let $m,p$ be odd and $n$ be even or $m,p$ be even and $n$ be odd.
   \begin{enumerate}
    \item If $m+p\equiv 0$ (mod 4), then $\gamma_s'(K_{m,n,p})=(m+p)/2+n+1$.
    \item If $m+p\equiv 2$ (mod 4), then $\gamma_s'(K_{m,n,p})=(m+p)/2+n$.
   \end{enumerate}

\item[{\bf E.}] Let $n,p$ be odd and $m$ be even or $n,p$ be even and $m$ be odd.
   \begin{enumerate}
    \item If $n+p\equiv 0$ (mod 4), then $\gamma_s'(K_{m,n,p})=(n+p)/2+m+1$.
    \item If $n+p\equiv 2$ (mod 4), then $\gamma_s'(K_{m,n,p})=(n+p)/2+m$.
   \end{enumerate}
\end{itemize}
In addition, $\gamma_s'(K_{1,1,1})=1$ and $\gamma_s'(K_{2,3,5})=5$.
}
\end{theorem}

In this paper, we find the signed edge
domination number of the complete tripartite
graph $K_{m,n,p}$, when $m,n\geq 1$ and $p\geq m+n$.
In Section \ref{SEC2}, we present some crucial results which will be employed
in the rest of this paper.
In Section \ref{SEC3}, we prove that if $f$ is a $\gamma_s'(K_{m,n,p})$-function,
then $f(w) \geq -1$ for every vertex $w$ in the largest partite set of $K_{m,n,p}$.
In Section \ref{SEC4}, we present general constructions for the SEDFs of
$K_{m,n,p}$ with minimum weight.
In Section \ref{SEC5}, we calculate the signed edge domination numbers of
$K_{1,n,p}$ and $K{2,2,p}$, where $p$ is even. These cases do not follow the constructions given in Section \ref{SEC4}. In addition, we notice that $\gamma_s'(K_{1,n,n+3})=2n+3$ when $n$ is odd. So there is an infinite family of graphs which achieve the upper bound given in Xu's conjecture (see \cite{Xu1}).
The main theorem of this paper is presented in Section \ref{SEC6}.

\section{Preliminary results}\label{SEC2}
Let $f$ be a SEDF of $G$.
An edge $e\in E(G)$ is called a {\em negative edge} ({\em positive edge})
if $f(e)=-1$ ($f(e)=1$). Let $uv$ be an edge of $G$ and
suppose that $x,y$ are the number of negative edges at vertices $u$ and $v$,
respectively. Then
\begin{equation}\label{eq1}
f[uv]=\deg(u)+\deg(v)-2x-2y-f(uv)\geq 1.
\end{equation}
Hence, if $e=uv$, then $f(u)+f(v)\geq 0$ for every edge $e\in E(G)$. In addition, if
$f(u)+f(v)=0$ or $1$, then $f(e)=-1$.

The following result can be found in \cite{CK}. Since there are typos in the proof given in \cite{CK}
we modify the proof and present it here.

\begin{lemma}\label{equal_minus}
{\rm Let $G$ be a graph, $u,v\in V(G)$ and
$N(u)\setminus\{v\}=N(v)\setminus\{u\}$. Let $f$ be a SEDF of
$G$. Then there exists a SEDF of $G$, say $g$, with $w(g)=w(f)$
such that the difference between the number of negative edges at $u$ and at $v$ is at most 1.}
\end{lemma}

\begin{proof}
Let $V(G)=\{v_1=u,v_2=v,v_3,\ldots,v_n\}$ and let $x_i$, $1\leq
i\leq n$, be the number of edges $e$ at $v_i$ with $f(e)=-1$.

If $x_1\leq x_2-2$, then there exists a vertex $v_{\ell}$ such that
$f(v_1v_{\ell})=1$ and $f(v_2v_{\ell})=-1$ for some
$\ell\in\{3,4,\ldots,n\}$. Define $g:E(G)\rightarrow \{-1,1\}$ by
$g(v_1v_{\ell})=-1$, $g(v_2v_{\ell})=1$ and $g(e)=f(e)$ for $e\in
E(G)\setminus\{v_1v_{\ell},v_2v_{\ell}\}$. Let $y_i$, $1\leq i\leq
n$, be the number of edges $e$ at $v_i$ with $g(e)=-1$. Obviously,
$y_1=x_1+1$, $y_2=x_2-1$ and $y_i=x_i$ for $3\leq i\leq n$. In
addition, $w(g)=w(f)$. We prove that $g$ is a SEDF of $G$.

If $v_1v_j\in E(G)$ and $j\not\in\{1,2,\ell\}$
$$\begin{array}{lclr}
g[v_1v_j]&=& \deg(v_1)+\deg(v_j)-2y_1-2y_j-g(v_1v_j)&\\
&=&\deg(v_1)+\deg(v_j)-2x_1-2-2x_j-f(v_1v_j)&\\
&\geq &\deg(v_2)+\deg(v_j)-2x_2+2-2x_j-f(v_1v_j)&\\
&\geq &\deg(v_2)+\deg(v_j)-2x_2-2x_j-f(v_2v_j)&\\
&\geq &k&\hfill{\mbox { by (\ref{eq1})}},\\
\end{array}$$
$$\begin{array}{lclr}
g[v_2v_j]&=& \deg(v_2)+\deg(v_j)-2y_2-2y_j-g(v_2v_j)&\\
&=&\deg(v_2)+\deg(v_j)-2x_2+2-2x_j-f(v_2v_j)&\\
&\geq &     k+2&\hfill {\mbox { by (\ref{eq1})}}.
\end{array}$$
Similarly, for $e\in
E(G)\setminus\{v_1v_j,v_2v_j\mid j\not\in\{1,2,\ell\}\}$,
we obtain $g[e]=f[e]\geq k$. In
addition, if $v_1v_2\in E(G)$, then
$$\begin{array}{lclr}
g[v_1v_2]&=& \deg(v_1)+\deg(v_2)-2y_1-2y_2-g(v_1v_2)&\\
&=&\deg(v_1)+\deg(v_2)-2x_1-2x_2-f(v_1v_2)&\\
&\geq &k&\hfill {\mbox { by (\ref{eq1})}}.
\end{array}$$
Hence, $g$ is a SEDF of $G$. If $g$ satisfies the required
condition, the proof is complete. Otherwise, by repeating this
process we can obtain the required function.
\end{proof}

\begin{corollary}\label{equal_minus_multipartite}
{\rm Let $G$ be a complete multipartite graph. There exists a
$\gamma_{sk}'(G)$-function $f$ such that the difference between the
number of negative edges at every two vertices in the same
partite set is at most 1.}
\end{corollary}

\section{SEDFs of complete tripartite graphs with vertices of negative weight}\label{SEC3}

\noindent Consider the complete tripartite graph $K_{m,n,p}$ with partite sets $U,V$ and $W$.
Throughout this section we assume $|U| =m$, $|V|=n$ and $|W|=p$, where $1\leq m\leq n\leq p$. In this section, we study SEDFs $f$ of $K_{m,n,p}$ such that
$f(w)<0$ for some $w\in W$.

\begin{lemma}\label{eee.ooo}
{\rm
Let $m,n,p$ be all even or all odd and $1\leq m\leq n\leq p$.
If $f$ is a $\gamma_s'(K_{m,n,p})$-function, then
$f(w)\geq 0$ for every vertex of $w\in W$.
}
\end{lemma}

\begin{proof}
The proof is by contradiction.
By Corollary \ref{equal_minus_multipartite}, we may assume that the difference between the
number of negative edges at every two vertices in the same
partite set of $K_{m,n,p}$ is at most 1.
Assume $f(w)=-2k$, where $1\leq k\leq  (m+n)/2$, for some $w\in W$,
and $f(w')\geq -2k$ for all $w'\in W$. Then there are $(m+n+2k)/2$ negative edges
and $(m+n-2k)/2$ positive edges at $w$. Therefore the weight of $(m+n+2k)/2$ vertices
in $U\cup V$ must be at least $2k$ and the weight of $(m+n-2k)/2$ vertices
in $U\cup V$ must be at least $2k+2$.
Since $m\leq n$ and $f$ is a $\gamma_s'(K_{m,n,p})$-function, we can assume $f(u)=2k$ for
every $u\in U$. So there are $(n+p-2k)/2$ negative edges at every vertex $u\in U$.
Let $U\cup V_1$, where $V_1\subseteq V$, consist of vertices of weight
$2k$ and let $V_2=V\setminus V_1$ consist of vertices of weight $2k+2$.
Since $(m+n+2k)/2 > m$ and $f$ is a $\gamma_s'(K_{m,n,p})$-function, it follows that there is a vertex
$v\in V_1$ of weight $2k$. Indeed, $|V_1|= (n-m+2k)/2$.
Therefore there are $(m+p-2k)/2$ negative edges at $v$.
Let $W_1\subseteq W$ consist of vertices of weight $-2k$. Since every vertex in
$V_1$ must be joined to every vertex in $W_1$ with a negative edge by (\ref{eq1}), it follows that
$|W_1|\leq (m+p-2k)/2$. Hence, $|W\setminus W_1|\geq (p-m+2k)/2$ and
every vertex in this set has weight $-2k+2$.
Note that $(n+p-2k)/2-(m+p-2k)/2=(n-m)/2$. Let $W_2$ be a subset
of $W\setminus W_1$ with $(n-m)/2$ vertices and the edges between $W_2$ and
$U$ are all negative edges.
Let $W_3=W\setminus (W_1\cup W_2)$. Then
$$|W_3\mid\geq p-[(m+p-2k)/2+(n-m)/2]=(p-n+2k)/2\geq 1.$$
Now let $w'\in W_3$.
Then the edges between $w'$ and $U\cup V_1$ are all positive edges. Therefor the number of
negative edges at $w'$ is at most $(n+m-2k)/2$. On the other hand,
for every vertex $w\in W_1$ there are $(n+m+2k)/2$ negative edges. Since $k\geq 1$,
the difference between the number of negative edges at $w$ and at $w'$ is $2k\geq 2$, which is a contradiction.
\end{proof}

The proof of the following result is similar to the proof of Lemma \ref{eee.ooo}.

\begin{lemma}\label{eeo.ooe}
Let $m,n$ be even and $p$ be odd or $m,n$ be odd and $p$ be even, where $1\leq m\leq n<p$.
If $f$ is a $\gamma_s'(K_{m,n,p})$-function, then
$f(w)\geq 0$ for every vertex of $w\in W$.
\end{lemma}

\begin{proof}
The proof is by contradiction.
Let $m,n$ be even and $p$ be odd. (The case $m,n$ odd and $p$ even is similar.)
By Corollary \ref{equal_minus_multipartite}, we may assume that the difference between the
number of negative edges at every two vertices in the same
partite set of $K_{m,n,p}$ is at most 1.
Assume $f(w)=-2k$, where $1\leq k\leq  (m+n)/2$, for some $w\in W$,
and $f(w')\geq -2k$ for all $w'\in W$. Then there are $(m+n+2k)/2$ negative edges
and $(m+n-2k)/2$ positive edges at $w$. Therefore the weight of $(m+n+2k)/2$ vertices
in $U\cup V$ must be at least $2k+1$ and the weight of $(m+n-2k)/2$ vertices
in $U\cup V$ must be at least $2k+3$.
Since $m\leq n$ and $f$ is a $\gamma_s'(K_{m,n,p})$-function, we can assume $f(u)=2k+1$ for
every $u\in U$. So there are $(n+p-2k-1)/2$ negative edges at every vertex $u\in U$.
Let $U\cup V_1$, where $V_1\subseteq V$, consist of vertices of weight
$2k+1$ and let $V_2=V\setminus V_1$ consist of vertices of weight $2k+3$.
Since $(m+n+2k)/2 > m$ and $f$ is a $\gamma_s'(K_{m,n,p})$-function, it follows that there is a vertex
$v\in V_1$ of weight $2k+1$. Indeed, $|V_1|= (n-m+2k)/2$.
Therefore there are $(m+p-2k-1)/2$ negative edges at $v$.
Let $W_1\subseteq W$ consist of vertices of weight $-2k$. Since every vertex in
$V_1$ must be joined to every vertex in $W_1$ with a negative edge by (\ref{eq1}), it follows that
$|W_1|\leq (m+p-2k-1)/2$. Hence, $|W\setminus W_1|\geq (p-m+2k+1)/2$ and
every vertex in this set has weight $-2k+2$.
Note that $(n+p-2k-1)/2-(m+p-2k-1)/2=(n-m)/2$. Let $W_2$ be a subset
of $W\setminus W_1$ with $(n-m)/2$ vertices and the edges between $W_2$ and
$U$ are all negative edges.
Let $W_3=W\setminus (W_1\cup W_2)$. Then
$$|W_3\mid\geq p-[(m+p-2k-1)/2+(n-m)/2]=(p-n+2k+1)/2\geq 1.$$
Now let $w'\in W_3$.
Then the edges between $w'$ and $U\cup V_1$ are all positive edges. Therefor the number of
negative edges at $w'$ is at most $(n+m-2k)/2$. On the other hand,
for every vertex $w\in W_1$ there are $(n+m+2k)/2$ negative edges. Since $k\geq 1$,
the difference between the number of negative edges at $w$ and at $w'$ is $2k\geq 2$, which is a contradiction.
\end{proof}

\begin{lemma}\label{oee.eoe}
{\rm
Let $m$ be odd and $n,p$ be even or $m,p$ be even and $n$ be odd, where $1\leq m< n\leq p$.
If $f$ is a $\gamma_s'(K_{m,n,p})$-function, then
$f(w)\geq -1$ for every vertex of $w\in W$.
}
\end{lemma}

\begin{proof}
The proof is by contradiction.
Let $m$ be odd and $n,p$ be even. (The case $m,p$ even and $n$ odd is similar.)
By Corollary \ref{equal_minus_multipartite}, we may assume that the difference between the
number of negative edges at every two vertices in the same
partite set of $K_{m,n,p}$ is at most 1.
Assume $f(w)=-2k-1$, where $1\leq k\leq  (m+n-1)/2$, for some $w\in W$,
and $f(w')\geq -2k-1$ for all $w'\in W$. Then there are $(m+n+2k+1)/2$ negative edges
and $(m+n-2k-1)/2$ positive edges at $w$. Therefore the weight of $(m+n+2k+1)/2$ vertices
in $U\cup V$ must be at least $2k+1$ and the weight of $(m+n-2k-1)/2$ vertices
in $U\cup V$ must be at least $2k+3$. Since $m<n$ and $f$ is a $\gamma_s'(K_{m,n,p})$-function,
we can assume $f(u)=2k+2$ for
every vertex $u\in U$. So there are $(n+p-2k-2)/2$ negative edges at every vertex $u\in U$.
Let $V_1\subset V$ consist of vertices of weight
$2k+1$ and let $V_2=V\setminus V_1$ consist of vertices of weight $2k+3$.
Since $(m+n+2k+1)/2 > m$ and $f$ is a $\gamma_s'(K_{m,n,p})$-function, it follows that there is a vertex
$v\in V_1$ of weight $2k+1$. Indeed, $|V_1|= (n-m+2k+1)/2$.
Therefore there are $(m+p-2k-1)/2$ negative edges at $v$.
Let $W_1\subseteq W$ consist of vertices of weight $-2k-1$. Since every vertex in
$V_1$ must be joined to every vertex in $W_1$ with a negative edge by (\ref{eq1}), it follows that
$|W_1|\leq (m+p-2k-1)/2$. Hence, $|W\setminus W_1|\geq (p-m+2k+1)/2$ and
every vertex in this set has weight $-2k+1$.

Note that $(n+p-2k-2)/2-(m+p-2k-1)/2=(n-m-1)/2$. Let $W_2$ be a subset
of $W\setminus W_1$ with $(n-m-1)/2$ vertices and the edges between $W_2$ and
$U$ are all negative edges. Let $W_3=W\setminus (W_1\cup W_2)$. Then
$$|W_3\mid\geq p-[(m+p-2k-1)/2+(n-m-1)/2]=(p-n+2k+2)/2\geq 2.$$
Now let $w'\in W_3$.
Then the edges between $w'$ and $U\cup V_1$ are all positive edges. Therefor the number of
negative edges at $w'$ is at most $(n+m-2k-1)/2$. On the other hand,
for every vertex $w\in W_1$ there are $(n+m+2k+1)/2$ negative edges. Since $k\geq 1$,
the difference between the number of negative edges at $w$ and at $w'$ is $2k+1\geq 3$, which is a contradiction.
\end{proof}

The proof of the following result is similar to the proof of Lemma \ref{oee.eoe}.

\begin{lemma}\label{oeo.eoo}
Let $m,p$ be odd and $n$ be even or $m$ be even and $n,p$ be odd, where $1\leq m<n\leq p$.
If $f$ is a $\gamma_s'(K_{m,n,p})$-function, then
$f(w)\geq -1$ for every vertex of $w\in W$.
\end{lemma}

\begin{proof}
The proof is by contradiction.
Let $m, p$ be odd and $n$ be even. (The case $m$ even and $n,p$ odd is similar.)
By Corollary \ref{equal_minus_multipartite}, we may assume that the difference between the
number of negative edges at every two vertices in the same
partite set of $K_{m,n,p}$ is at most 1.
Assume $f(w)=-2k-1$, where $1\leq k\leq  (m+n-1)/2$, for some $w\in W$,
and $f(w')\geq -2k-1$ for all $w'\in W$. Then there are $(m+n+2k+1)/2$ negative edges
and $(m+n-2k-1)/2$ positive edges at $w$. Therefore the weight of $(m+n+2k+1)/2$ vertices
in $U\cup V$ must be at least $2k+1$ and the weight of $(m+n-2k-1)/2$ vertices
in $U\cup V$ must be at least $2k+3$. Since $m<n$ and $f$ is a $\gamma_s'(K_{m,n,p})$-function,
we can assume $f(u)=2k+1$ for
every vertex $u\in U$. So there are $(n+p-2k-1)/2$ negative edges at every vertex $u\in U$.
Let $V_1\subset V$ consist of vertices of weight
$2k+2$ and let $V_2=V\setminus V_1$ consist of vertices of weight $2k+4$.
Since $(m+n+2k+1)/2 > m$ and $f$ is a $\gamma_s'(K_{m,n,p})$-function, it follows that there is a vertex
$v\in V_1$ of weight $2k+2$. Indeed, $|V_1|= (n-m+2k+1)/2$.
Therefore there are $(m+p-2k-2)/2$ negative edges at $v$.
Let $W_1\subseteq W$ consist of vertices of weight $-2k-1$. Since every vertex in
$V_1$ must be joined to every vertex in $W_1$ with a negative edge by (\ref{eq1}), it follows that
$|W_1|\leq (m+p-2k-2)/2$. Hence, $|W\setminus W_1|\geq (p-m+2k+2)/2$ and
every vertex in this set has weight $-2k+1$.

Note that $(n+p-2k-1)/2-(m+p-2k-2)/2=(n-m+1)/2$. Let $W_2$ be a subset
of $W\setminus W_1$ with $(n-m+1)/2$ vertices and the edges between $W_2$ and
$U$ are all negative edges. Let $W_3=W\setminus (W_1\cup W_2)$. Then
$$|W_3\mid\geq p-[(m+p-2k-2)/2+(n-m+1)/2]=(p-n+2k+1)/2\geq 1.$$
Now let $w'\in W_3$.
Then the edges between $w'$ and $U\cup V_1$ are all positive edges. Therefor the number of
negative edges at $w'$ is at most $(n+m-2k-1)/2$. On the other hand,
for every vertex $w\in W_1$ there are $(n+m+2k+1)/2$ negative edges. Since $k\geq 1$,
the difference between the number of negative edges at $w$ and at $w'$ is $2k+1\geq 3$, which is a contradiction.
\end{proof}

\section{The SEDN of $K_{m,n,p}$}\label{SEC4}
Consider the complete tripartite graph $K_{m,n,p}$ whose partite sets are $U,V$ and $W$.
Throughout this section we assume $|U|=m$, $|V|=n$ and $|W\mid=p$, where $m,n$ and $p$ are positive integers, $m\leq n$ and $p\geq m+n$
In this section we compute the signed edge domination number of $K_{m,n,p}$, where $m\geq 2$
and $(m,n)\neq(2,2)$ if $p$ is odd.

\begin{theorem}\label{p>m+n.eee}
{\rm Let $m,n$ and $p$ be even and $p \geq m+n$.
Then $\gamma_s'(K_{m,n,p})=m+n$.}
\end{theorem}

\begin{proof}
Consider the graph $K_{m,n,p}$ whose partite sets are $U,V$ and $W$.
By assumption
$$m(n+p-2)/2+n(m+p-2)/2-p(m+n)/2=mn-m-n$$
is even.
First we label $(mn-m-n)/2$ edges between $U$ and $V$ with $-1$ in the following way.
Label an edge $uv$, where $u\in U$ and $v\in V$, with $-1$ if
\begin{enumerate}
\item the total number of negative edges between $U$ and $V$ is less than
$(mn-m-n)/2$,

\item the number of negative edges at $u$ is less than $(n+p-2)/2$,

\item the number of negative edges at $v$ is less than $(m+p-2)/2$,

\item the number of negative edges at $u$ is less than or equal to
the number of negative edges at $u'$ for every $u'\in (U\setminus\{u\})$, and

\item the number of negative edges at $v$ is less than or equal to
the number of negative edges at $v'$ for every $v'\in (V\setminus\{v\})$.
\end{enumerate}

\noindent Then we label $p(m+n)/2$ edges between $U\cup V$ and $W$ with $-1$ in
a similar fashion described above.
An edge $rw$, where $r\in (U\cup V)$ and $w\in W$ is labelled by $-1$ if
\begin{enumerate}
\item the number of negative edges between $U\cup V$ and $W$ is less than
$p(m+n)/2$,

\item the number of negative edges at $r$ is less than $(n+p-2)/2$ if $r\in U$,

\item the number of negative edges at $r$ is less than $(m+p-2)/2$ if $r\in V$,

\item the number of negative edges at $w$ is less than $(m+n)/2$,

\item the number of negative edges at $r$ is less than or equal to
the number of negative edges at $u$ for every $u\in (U\setminus\{r\})$ if $r\in U$,

\item the number of negative edges at $r$ is less than or equal to
the number of negative edges at $v$ for every $v\in (V\setminus\{r\})$ if $r\in V$,

\item the number of negative edges at $w$ is less than or equal to
the number of negative edges at $w'$ for every $w'\in (W\setminus\{w\})$.
\end{enumerate}

\noindent
Then there are exactly $((p+n)/2)-1$, $((p+m)/2)-1$ and $(m+n)/2$
negative edges at every vertex in $U, V$ and $W$, respectively.
Label the remaining edges of $K_{m,n,p}$ by $+1$. Then
all vertices in $U\cup V$ have weight $2$ and all the vertices in $W$
have weight zero.
Hence, this labeling defines a SEDF $f$ of $K_{m,n,p}$ by (\ref{eq1}), and $\omega(f)=m+n$.
Note that since the weight of every vertex of $W$ is zero, no vertices in $U\cup V$ can
have weight zero by (\ref{eq1}).
Now by Lemma \ref{eee.ooo} and the facts that $f(W)=0$ and $f(r)=2$ for every $r\in U\cup V$,
it follows that $\gamma_s'(K_{m,n,p})=m+n$.
\end{proof}

\begin{theorem}\label{p>=m+n+1.ooo}
{\rm Let $m,n$ and $p$ be odd, $m,n\geq 3$ and $p\geq m+n+1$.
Then $\gamma_s'(K_{m,n,p})=m+n+1$.
}
\end{theorem}

\begin{proof}
Consider the graph $K_{m,n,p}$ whose partite sets are $U,V$ and $W$.
By assumption,
$$m(n+p-2)/2+n(m+p-2)/2-p(m+n)/2=mn-m-n$$
is odd. In addition,
$(n+p-2)/2$, $(m+p-2)/2$ and $(m+n)/2$ are
all odd, or two are even and one is odd. Hence, there is no graph whose
$m$ vertices have degree $(n+p-2)/2$, $n$ vertices have degree $(n+p-2)/2$ and
$p$ vertices have degree $(m+n)/2$.

\noindent On the other hand,
$$\begin{array}{l}
m(n+p-2)/2+n(m+p-2)/2-(p-1)(m+n)/2\\
-(m+n-2)/2 =mn-m-n+1,
\end{array}$$
is an even number.
We label $(mn-m-n+1)/2$ edges between
$U$ and $V$ and $(p-1)(m + n)/2 + (m+n-2)/2$ edges between $U\cup V$ and
$W$ with $-1$  in a similar fashion
described in Theorem \ref{p>m+n.eee}.
Then there are $(n+p-2)/2$ negative edges at each vertex of $U$, $(m+p-2)/2$ negative
edges at each vertex of $V$ and $(m+n)/2$ negative edges at each vertex of $W$ except one vertex
which is incident with $(m+n-2)/2$ negative edges.
We label the remaining edges of $K_{m,n,p}$ with $+1$.
Then the weight of vertices in $U\cup V$ are all $2$ and the weight of vertices in $W$
are all zero except one vertex of $W$ whose weight is $2$.
Hence, this labeling defines a SEDF $f$ of $K_{m,n,p}$ by (\ref{eq1}), and $\omega(f)=m+n+1$.

Note that since the weight of every vertex of $W$ is zero, no vertices in $U\cup V$ can
have weight zero by (\ref{eq1}).
Now by Lemma \ref{eee.ooo} and the facts that $f(W)=2$ and $f(r)=2$ for every $r\in U\cup V$,
it follows that $\gamma_s'(K_{m,n,p})=m+n+1$.
\end{proof}

\begin{theorem}\label{p>m+n.ooe}
{\rm Let $m, n\geq 3$ be odd, $p$ be even and $p\geq m+n$.
Then
\begin{enumerate}
\item $\gamma_s'(K_{m,n,p})=\dfrac{3m+3n+2}{2}$ if $m+n\equiv 0 \pmod 4$,

\item $\gamma_s'(K_{m,n,p})=\dfrac{3m+3n}{2}$ if $m+n\equiv 2 \pmod 4$.
\end{enumerate}
}
\end{theorem}

\begin{proof}
Consider the graph $K_{m,n,p}$ whose partite sets are $U,V$ and $W$.

\noindent{\bf Case 1.}\quad $m+n\equiv 0 \pmod 4$.\\
By assumption,
$$\begin{array}{l}
(m-1)(n+p - 3)/2+(n+p-5)/2 + n (m + p - 3)/2 \\
- p (m + n)/2=mn-1-(3m+3n)/2\end{array}$$
is even.
Label $(1/2)[mn-1-(3m+3n)/2]$ edges between $U$ and $V$ and $p(m+n)/2$ edges between
$U\cup V$ and $W$ with $-1$ in a similar fashion
described in the proof of Theorem \ref{p>m+n.eee}.
Then every vertex in $U$ is incident with $(n+p-3)/2$ negative edges except one vertex which is incident with
$(n+p-5)/2$ negative edges. Every vertices in $V$ is incident with $(m+p-3)/2$ negative edges and
every vertex in $W$ is incident with $(m+n)/2$ negative edges.
Label the remaining edges of $K_{n,m,p}$ by $+1$.
Then the weight of vertices in $U\cup V$ are all $3$ except one vertex of $U$
whose weight is $5$, and the weight of vertices in $W$ are all zero.
Hence, this labeling defines a SEDF $f$ with $w(f)=(3m+3n+2)/2$.

Note that since the weight of every vertex of $W$ is zero, no vertices in $U\cup V$ can
have weight one by (\ref{eq1}).
Now by Lemma \ref{eeo.ooe} and the facts that $f(W)=0$ and $f(r)=3$ for every $r\in U\cup V$ except one
vertex which has weight 5, it follows that $\gamma_s'(K_{m,n,p})=(3m+3n+2)/2$.
\vspace{3mm}

\noindent{\bf Case 2.}\quad $m+n\equiv 2 \pmod 4$.\\
By assumption, $$m (p + n - 3)/2 + n (m + p - 3)/2 - p (m + n)/2=mn-(3m+3n)/2$$
is even.
Label $(1/2)[mn-(3m+3n)/2]$ edges between $U$ and $V$ and $p(m+n)/2$ edges between
$U\cup V$ and $W$ with $-1$ in a similar fashion
described in the proof of Theorem \ref{p>m+n.eee}. Then every
vertex in $U$ is incident with $(n+p-3)/2$ negative edges,
every vertex in $V$ is incident with $(m+p-3)/2$ negative edges and
every vertex in $W$ is incident with $(m+n)/2$ negative edges.
Label the remaining edges of $K_{n,m,p}$ with $+1$.
Then the weight of vertices in $U\cup V$ are all $3$
and the weight of vertices in $W$ are all zero.
Hence, this labeling defines a SEDF $f$ with $w(f)=(3m+3n)/2$.

Note that since the weight of every vertex of $W$ is zero, no vertices in $U\cup V$ can
have weight one by (\ref{eq1}).
Now by Lemma \ref{eeo.ooe} and the facts that $f(W)=0$ and $f(r)=3$
for every $r\in U\cup V$, it follows that $\gamma_s'(K_{m,n,p})=(3m+3n)/2$.
\end{proof}

\begin{theorem}\label{p>=m+n+1.eeo}
{\rm
Let $m$ and $n$ be even, $(m,n)\neq (2,2)$, $p$ be odd and $p\geq m+n+1$.
Then
\begin{enumerate}
\item $\gamma_s'(K_{m,n,p})=\dfrac{3m+3n}{2}$ if $m+n\equiv 0 \pmod 4$,

\item $\gamma_s'(K_{m,n,p})=\dfrac{3m+3n+2}{2}$ if $m+n\equiv 2\pmod 4$.
\end{enumerate}
}
\end{theorem}

\begin{proof}
Consider the graph $K_{m,n,p}$ whose partite sets are $U,V$ and $W$.

\noindent{\bf Case 1.}\quad $m+n\equiv 0\pmod 4$.\\
By assumption,
$$\begin{array}{l}
m(n+p-3)/2+n(m+p-3)/2-p(m+n)/2=mn-(3m+3n)/2\\
\end{array}$$
is even.
Label $(1/2)[mn-3(m+n)/2]$ edges between $U$ and $V$ and $p(m+n)/2$ edges between
$U\cup V$ and $W$ with $-1$ as described in the proof of Theorem \ref{p>m+n.eee}.
Then every vertex in $U$ is incident with $(n+p-3)/2$ negative edges,
every vertex in $V$ is incident with $(m+p-3)/2$ negative edges and every vertex
in $W$ is incident with $(m+n)/2$ negative edges.
Label the remaining edges of $K_{n,m,p}$ by $+1$.
Then the weight of vertices in $U\cup V$ are all $3$
and the weight of vertices in $W$ are all zero.
Hence, this labeling defines a SEDF $f$ with $w(f)=(3m+3n)/2$.

Note that since the weight of every vertex of $W$ is zero, no vertices in $U\cup V$ can
have weight one by (\ref{eq1}).
Now by Lemma \ref{eeo.ooe} and the facts that $f(W)=0$ and $f(r)=3$
for every $r\in U\cup V$, it follows that $\gamma_s'(K_{m,n,p})=(3m+3n)/2$.
\vspace{3mm}

\noindent{\bf Case 2.}\quad $m+n\equiv 2\pmod 4$.\\
By assumption,
$$\begin{array}{l}
(m-1)(n+p-3)/2+(n+p-5)/2+n(m+p-3)/2\\-p(m+n)/2
=mn-1-(3m+3n)/2\\
\end{array}$$
is even.
Label $(1/2)[mn-1-3(m+n)/2]$ edges between $U$ and $V$ and $p(m+n)/2$ edges between
$U\cup V$ and $W$ with $-1$ as described in the proof of Theorem \ref{p>m+n.eee}. Then
every vertex in $U$ is incident with $(n+p-3)/2$ negative edges except one vertex which is
incident with $(n+p-5)/2$ negative edges, every vertex in
$V$ is incident with $(m+p-3)/2$ negative edges, and every vertex
in $W$ is incident with $(m+n)/2$ negative edges.
Label the remaining edges of $K_{n,m,p}$ with $+1$.
Then the weight of vertices in $U\cup V$ are all $3$ except one vertex whose weight is 5 and
the weight of vertices in $W$ are all zero.
Hence, this labeling defines a SEDF $f$ with $w(f)=(3m+3n+2)/2$.

Note that since the weight of every vertex of $W$ is zero, no vertices in $U\cup V$ can
have weight one by (\ref{eq1}).
Now by Lemma \ref{eeo.ooe} and the facts that $f(W)=0$, $f(u)=3$
for every vertex $u$ in $U$ except one vertex whose weight is 5,
and $f(v)=3$ for every $v\in V$, it follows that $\gamma_s'(K_{m,n,p})=(3m+3n+2)/2$.
\end{proof}

\begin{theorem}\label{p>=m+n+1.oee}
{\rm
Let $m$ be odd, $n,p$ be even, $3\leq m<n$ and $p\geq m+n+1$.
Then
\begin{enumerate}
\item $\gamma_s'(K_{m,n,p})=\dfrac{3m+2n+1}{2}$ if $m\equiv 1\pmod 4$,

\item $\gamma_s'(K_{m,n,p})=\dfrac{3m+2n-1}{2}$ if $m\equiv 3\pmod 4$
\end{enumerate}
}
\end{theorem}

\begin{proof}
Consider the graph $K_{m,n,p}$ whose partite sets are $U,V$ and $W$.

\noindent{\bf Case 1.}\quad $m\equiv 1\pmod 4$.\\
By assumption,
$$\begin{array}{l}
m(n+p-2)/2+((n-m-1)/2)(m+p-1)/2\\+((n+m+1)/2)(m+p-3)/2-(p/2)(m+n+1)/2\\
-(p/2)(m+n-1)/2=mn-n-(3m+1)/2
\end{array}$$
is even.
Partition $V$ into $V_1$ and $V_2$ with $|V_1\mid=(n-m-1)/2$ and $|V_2\mid=(n+m+1)/2$.
Also partition $W$ into $W_1$ and $W_2$ with $|W_1\mid=|W_2\mid$.
In a similar fashion described in the proof of Theorem \ref{p>m+n.eee},
label $(1/2)[mn-n-(3m+1)/2]$ edges between $U$ and $V$ and
$(p/2)(m+n+1)/2+(p/2)(m+n-1)/2$ edges between $U\cup V$ and $W$ with $-1$
such that the edges between $V_1$ and $W_1$ are all negative edges. In addition,
every vertex in $U$ is incident with $(n+p-2)/2$ negative edges,
every vertex in $V_1, V_2$ is incident with $(m+p-1)/2$ and $(m+p-3)/2$ negative edges, respectively,
and every vertex in $W_1, W_2$ is incident with $(m+n+1)/2$ and $(m+n-1)/2$ negative edges, respectively.
Label the remaining edges of $K_{n,m,p}$ by $+1$.
Then the weight of vertices in $U$ are all $2$,
the weight of vertices in $V_1, V_2$ are all $1,3$, respectively,
and the weight of vertices in $W_1$ are all $-1$ and in $W_2$ are $+1$.
Hence, this labeling defines a SEDF $f$ with $w(f)=(3m+2n+1)/2$.

Note that since the weight of some vertices in $W$ is $-1$, no vertices in $U\cup V$ can
have weight zero or $-1$ by (\ref{eq1}).
Now by Lemma \ref{oee.eoe} and the facts that $f(W)=0$, $f(u)=2$
for every $u\in U$, $f(v)=1$ for every $v\in V_1$ and $f(v)=3$ for every $v\in V_2$, it follows that $\gamma_s'(K_{m,n,p})=(3m+2n+1)/2$.
\vspace{3mm}

\noindent{\bf Case 2.}\quad $m\equiv 3\pmod 4$.\\
By assumption,
$$\begin{array}{l}
m(n+p-2)/2+((n-m+1)/2)(m+p-1)/2\\+((n+m-1)/2)(m+p-3)/2-(p/2)(m+n+1)/2\\-(p/2)(m+n-1)/2=
mn-n-(3m-1)/2
\end{array}$$
is even.
Partition $V$ into $V_1$ and $V_2$ with $|V_1\mid=(n-m+1)/2$ and $|V_2\mid=(n+m-1)/2$.
Also partition, $W$ into $W_1$ and $W_2$ with $|W_1\mid=|W_2\mid$.
In a similar fashion described in the proof of Theorem \ref{p>m+n.eee},
label $(1/2)[mn-n-(3m-1)/2]$ edges between $U$ and $V$ and
$(p/2)(m+n+1)/2+(p/2)(m+n-1)/2$ edges between $U\cup V$ and $W$ with $-1$
such that the edges between $V_1$ and $W_1$ are all negative edges.
In addition, every vertex in $U$ is incident with $(n+p-2)/2$ negative edges,
every vertex in $V_1, V_2$ is incident with $(m+p-1)/2$ and $(m+p-3)/2$ negative edges, respectively,
and every vertex in $W_1, W_2$ is incident with $(m+n+1)/2$ and $(m+n-1)/2$ negative edges, respectively.
Label the remaining edges of $K_{n,m,p}$ by $+1$.
Then the weight of vertices in $U$ are all $2$,
the weight of vertices in $V_1, V_2$ are all $1,3$, respectively,
and the weight of vertices in $W_1$ are all $-1$ and in $W_2$ are $+1$.
Hence, this labeling defines a SEDF $f$ with $w(f)=(3m+2n-1)/2$.

Note that since the weight of some vertices in $W$ is $-1$, no vertices in $U\cup V$ can
have weight zero or $-1$ by (\ref{eq1}).
Now by Lemma \ref{oee.eoe} and the facts that $f(W)=0$, $f(u)=2$
for every $u\in U$, $f(v)=1$ for every $v\in V_1$ and $f(v)=3$ for every $v\in V_2$, it follows that $\gamma_s'(K_{m,n,p})=(3m+2n-1)/2$.
\end{proof}

\begin{theorem}\label{p>=m+n+1.eoe}
{\rm
Let $m,p$ be even, $n$ be odd, $m<n$ and $p\geq m+n+1$.
Then
\begin{enumerate}
\item $\gamma_s'(K_{m,n,p})=\dfrac{2m+3n+1}{2}$ if $n\equiv 1\pmod 4$,
\item $\gamma_s'(K_{m,n,p})=\dfrac{2m+3n-1}{2}$ if $n\equiv 3\pmod 4$.
\end{enumerate}
}
\end{theorem}

\begin{proof}
Consider the graph $K_{m,n,p}$ whose partite sets are $U,V$ and $W$.

\noindent{\bf Case 1.}\quad $n\equiv 1\pmod 4$.\\
By assumption,
$$\begin{array}{l}
m(n+p-1)/2+((n-m-1)/2)(m+p-2)/2\\+((n+m+1)/2)(m+p-4)/2-(p/2)(m+n+1)/2\\-(p/2)(m+n-1)/2=
mn-m-(3n+1)/2
\end{array}$$
is even.
Partition $V$ into $V_1$ and $V_2$ with $|V_1\mid=(n-m-1)/2$ and $|V_2\mid=(n+m+1)/2$.
Also partition, $W$ into $W_1$ and $W_2$ with $|W_1\mid=|W_2\mid$.
In a similar fashion described in the proof of Theorem \ref{p>m+n.eee},
label $(1/2)[mn-n-(3n+1)/2]$ edges between $U$ and $V$ and
$(p/2)(m+n+1)/2+(p/2)(m+n-1)/2$ edges between $U\cup V$ and $W$ with $-1$
such that the edges between $V_1$ and $W_1$ are all negative edges.
In addition, every vertex in $U$ is incident with $(n+p-1)/2$ negative edges,
every vertex in $V_1, V_2$ is incident with $(m+p-2)/2$ and $(m+p-4)/2$ negative edges, respectively,
and every vertex in $W_1, W_2$ is incident with $(m+n+1)/2$ and $(m+n-1)/2$ negative edges, respectively.
Label the remaining edges of $K_{n,m,p}$ by $+1$.
Then the weight of vertices in $U$ are all $1$,
the weight of vertices in $V_1, V_2$ are all $2,4$, respectively,
the weight of vertices in $W_1$ are all $-1$ and in $W_2$ are $+1$.
Hence, this labeling defines a SEDF $f$ with $w(f)=(2m+3n+1)/2$.

Note that since the weight of some vertices in $W$ is $-1$, no vertices in $U\cup V$ can
have weight zero or $-1$ by (\ref{eq1}).
Now by Lemma \ref{oee.eoe} and the facts that $f(W)=0$, $f(u)=1$
for every $u\in U$, $f(v)=2$ for every $v\in V_1$ and $f(v)=4$ for every $v\in V_2$, it follows that $\gamma_s'(K_{m,n,p})=(2m+3n+1)/2$.
\vspace{3mm}

\noindent{\bf Case 2.}\quad $n\equiv 3\pmod 4$.\\
By assumption,
$$\begin{array}{l}
m(n+p-1)/2+((n-m-1)/2)(m+p-2)/2\\+((n+m+1)/2)(m+p-4)/2-(p/2)(m+n+1)/2\\-(p/2)(m+n-1)/2=
mn-m-(3n-1)/2
\end{array}$$
is even.
Partition $V$ into $V_1$ and $V_2$ with $|V_1\mid=(n-m+1)/2$ and $|V_2\mid=(n+m-1)/2$.
Also partition, $W$ into $W_1$ and $W_2$ with $|W_1\mid=|W_2\mid$.
In a similar fashion described in the proof of Theorem \ref{p>m+n.eee},
label $(1/2)[mn-n-(3n+1)/2]$ edges between $U$ and $V$ and
$(p/2)(m+n+1)/2+(p/2)(m+n-1)/2$ edges between $U\cup V$ and $W$ with $-1$
such that the edges between $V_1$ and $W_1$ are all negative edges.
In addition, every vertex in $U$ is incident with $(n+p-1)/2$ negative edges,
every vertex in $V_1, V_2$ is incident with $(m+p-2)/2$ and $(m+p-4)/2$ negative edges, respectively,
and every vertex in $W_1, W_2$ is incident with $(m+n+1)/2$ and $(m+n-1)/2$ negative edges, respectively.
Label the remaining edges of $K_{n,m,p}$ by $+1$.
Then the weight of vertices in $U$ are all $1$,
the weight of vertices in $V_1, V_2$ are all $2,4$, respectively,
and the weight of vertices in $W_1$ are all $-1$ and in $W_2$ are $+1$.
Hence, this labeling defines a SEDF $f$ with $w(f)=(2m+3n-1)/2$.

Note that since the weight of some vertices in $W$ is $-1$, no vertices in $U\cup V$ can
have weight zero or $-1$ by (\ref{eq1}).
Now by Lemma \ref{oee.eoe} and the facts that $f(W)=0$, $f(u)=1$
for every $u\in U$, $f(v)=2$ for every $v\in V_1$ and $f(v)=4$ for every $v\in V_2$, it follows that $\gamma_s'(K_{m,n,p})=(2m+3n-1)/2$.
\end{proof}

\begin{theorem}\label{p>=m+n+1.oeo}
{\rm
Let $m,p$ be odd, $n$ be even, $3\leq m<n$, and $p\geq m+n$.
Then
\begin{enumerate}
\item $\gamma_s'(K_{m,n,p})=\dfrac{2m+3n}{2}$ if $n\equiv 0 \pmod 4$,

\item $\gamma_s'(K_{m,n,p})=\dfrac{2m+3n-2}{2}$ if $n\equiv 2 \pmod 4$,
\end{enumerate}
}
\end{theorem}

\begin{proof}
Consider the graph $K_{m,n,p}$ whose partite sets are $U,V$ and $W$.

\noindent{\bf Case 1.}\quad $n\equiv 0\pmod 4$.\\
By assumption,
$$\begin{array}{l}
m(n+p-1)/2+((n-m-1)/2)(m+p-2)/2\\+((n+m+1)/2)(m+p-4)/2-((p+1)/2)(m+n+1)/2\\-((p-1)/2)(m+n-1)/2
=mn-m-(3n+2)/2
\end{array}$$
is even.
Partition $V$ into $V_1$ and $V_2$ with $|V_1\mid=(n-m-1)/2$ and $|V_2\mid=(n+m+1)/2$.
Also partition $W$ into $W_1$ and $W_2$ with $|W_1\mid=|W_2\mid+1$.
In a similar fashion described in the proof of Theorem \ref{p>m+n.eee},
label $(1/2)[mn-m-(3n+2)/2]$ edges between $U$ and $V$ and
$((p+1)/2)(m+n+1)/2+((p-1)/2)(m+n-1)/2$ edges between $U\cup V$ and $W$ with $-1$
such that the edges between $V_1$ and $W_1$ are all negative edges.
In addition, every vertex in $U$ is incident with $(n+p-1)/2$ negative edges,
every vertex in $V_1, V_2$ is incident with $(m+p-2)/2$ and $(m+p-4)/2$ negative edges, respectively,
and every vertex in $W_1, W_2$ is incident with $(m+n+1)/2$ and $(m+n-1)/2$ negative edges, respectively.
Label the remaining edges of $K_{n,m,p}$ by $+1$.
Then the weight of vertices in $U$ are all $1$,
the weight of vertices in $V_1, V_2$ are all $2,4$, respectively,
and the weight of vertices in $W_1$ are $-1$ and in $W_2$ are $+1$.
Hence, this labeling defines a SEDF $f$ with $w(f)=(2m+3n)/2$.

Note that since the weight of some vertices in $W$ is $-1$, no vertices in $U\cup V$ can
have weight zero or $-1$ by (\ref{eq1}).
Now by Lemma \ref{oeo.eoo} and the facts that $f(W)=-1$, $f(u)=1$
for every $u\in U$, $f(v)=2$ for every $v\in V_1$ and $f(v)=4$ for every $v\in V_2$, it follows that $\gamma_s'(K_{m,n,p})=(2m+3n)/2$.
\vspace{3mm}

\noindent{\bf Case 2.}\quad $n\equiv 2\pmod 4$.\\
By assumption,
$$\begin{array}{l}
m(n+p-1)/2+((n-m+1)/2)(m+p-2)/2\\+((n+m-1)/2)(m+p-4)/2-((p+1)/2)(m+n+1)/2\\-((p-1)/2)(m+n-1)/2
=mn-m-3n/2
\end{array}$$
is even.
Partition $V$ into $V_1$ and $V_2$ with $|V_1\mid=(n-m+1)/2$ and $|V_2\mid=(n+m-1)/2$.
Also partition $W$ into $W_1$ and $W_2$ with $|W_1\mid=|W_2\mid+1$.
In a similar fashion described in the proof of Theorem \ref{p>m+n.eee},
label $(1/2)[mn-m-3n/2]$ edges between $U$ and $V$ and
$((p+1)/2)(m+n+1)/2+((p-1)/2)(m+n-1)/2$ edges between $U\cup V$ and $W$ with $-1$
such that the edges between $V_1$ and $W_1$ are all negative edges.
In addition, every vertex in $U$ is incident with $(n+p-1)/2$ negative edges,
every vertex in $V_1, V_2$ is incident with $(m+p-2)/2$ and $(m+p-4)/2$ negative edges, respectively,
and every vertex in $W_1, W_2$ is incident with $(m+n+1)/2$ and $(m+n-1)/2$ negative edges, respectively.
Label the remaining edges of $K_{n,m,p}$ by $+1$.
Then the weight of vertices in $U$ are all $1$,
the weight of vertices in $V_1, V_2$ are all $2,4$, respectively,
and the weight of vertices in $W_1$ are all $-1$ and in $W_2$ are $+1$.
Hence, this labeling defines a SEDF $f$ with $w(f)=(2m+3n-2)/2$.

Note that since the weight of some vertices in $W$ is $-1$, no vertices in $U\cup V$ can
have weight zero or $-1$ by (\ref{eq1}).
Now by Lemma \ref{oeo.eoo} and the facts that $f(W)=-1$, $f(u)=1$
for every $u\in U$, $f(v)=2$ for every $v\in V_1$ and $f(v)=4$ for every $v\in V_2$, it follows that $\gamma_s'(K_{m,n,p})=(2m+3n-2)/2$.
\end{proof}

\begin{theorem}\label{p>=m+n+1.eoo}
{\rm
Let $m$ be even, $n,p$ be odd, $m<n$, and $p\geq m+n$.
Then there is an SEDF $f$ of $K_{m,n,p}$ with
\begin{enumerate}
\item $\gamma_s'(K_{m,n,p})=\dfrac{3m+2n}{2}$ if $m\equiv 0 \pmod 4$,

\item $\gamma_s'(K_{m,n,p})=\dfrac{3m+2n-2}{2}$ if $m\equiv 2 \pmod 4$.
\end{enumerate}
}
\end{theorem}

\begin{proof}
Consider the graph $K_{m,n,p}$ whose partite sets are $U,V$ and $W$.

\noindent{\bf Case 1.}\quad $m\equiv 0\pmod 4$.\\
By assumption,
$$\begin{array}{l}
m(n+p-2)/2+((n-m+1)/2)(m+p-1)/2\\+((n+m-1)/2)(m+p-3)/2-((p-1)/2)(m+n+1)/2\\-((p+1)/2)(m+n-1)/2
=mn-n-(3m-2)/2
\end{array}$$
is even.
Partition $V$ into $V_1$ and $V_2$ with $|V_1\mid=(n-m+1)/2$ and $|V_2\mid=(n+m-1)/2$.
Also partition $W$ into $W_1$ and $W_2$ with $|W_1\mid=|W_2\mid-1$.
In a similar fashion described in the proof of Theorem \ref{p>m+n.eee},
label $(1/2)[mn-n-(3m-2)/2]$ edges between $U$ and $V$ and
$((p-1)/2)(m+n+1)/2+((p+1)/2)(m+n-1)/2$ edges between $U\cup V$ and $W$ with $-1$
such that the edges between $V_1$ and $W_1$ are all negative edges.
In addition, every vertex in $U$ is incident with $(n+p-2)/2$ negative edges,
every vertex in $V_1, V_2$ is incident with $(m+p-1)/2$ and $(m+p-3)/2$ negative edges, respectively,
and every vertex in $W_1, W_2$ is incident with $(m+n+1)/2$ and $(m+n-1)/2$ negative edges, respectively.
Label the remaining edges of $K_{n,m,p}$ by $+1$.
Then the weight of vertices in $U$ are all $2$,
the weight of vertices in $V_1, V_2$ are all $1,3$, respectively,
and the weight of vertices in $W_1$ are all $-1$ and in $W_2$ are $+1$.
Hence, this labeling defines a SEDF $f$ with $w(f)=(3m+2n)/2$.

Note that since the weight of some vertices in $W$ is $-1$, no vertices in $U\cup V$ can
have weight zero or $-1$ by (\ref{eq1}).
Now by Lemma \ref{oeo.eoo} and the facts that $f(W)=1$, $f(u)=2$
for every $u\in U$, $f(v)=1$ for every $v\in V_1$ and $f(v)=3$ for every $v\in V_2$, it follows that $\gamma_s'(K_{m,n,p})=(3m+2n)/2$.
\vspace{3mm}

\noindent{\bf Case 2.}\quad $m\equiv 2\pmod 4$.\\
By assumption,
$$\begin{array}{l}
m(n+p-2)/2+((n-m+1)/2)(m+p-1)/2\\+((n+m-1)/2)(m+p-3)/2-((p+1)/2)(m+n+1)/2\\-((p-1)/2)(m+n-1)/2
=mn-n-(3m)/2
\end{array}$$
is even.
Partition $V$ into $V_1$ and $V_2$ with $|V_1\mid=(n-m+1)/2$ and $|V_2\mid=(n+m-1)/2$.
Also partition $W$ into $W_1$ and $W_2$ with $|W_1\mid=|W_2\mid+1$.
In a similar fashion described in the proof of Theorem \ref{p>m+n.eee},
label $(1/2)[mn-n-3m/2]$ edges between $U$ and $V$ and
$((p+1)/2)(m+n+1)/2+((p-1)/2)(m+n-1)/2$ edges between $U\cup V$ and $W$ with $-1$
such that the edges between $V_1$ and $W_1$ are all negative edges.
In addition, every vertex in $U$ is incident with $(n+p-2)/2$ negative edges,
every vertex in $V_1, V_2$ is incident with $(m+p-1)/2$ and $(m+p-3)/2$ negative edges, respectively,
and every vertex in $W_1, W_2$ is incident with $(m+n+1)/2$ and $(m+n-1)/2$ negative edges, respectively.
Label the remaining edges of $K_{n,m,p}$ by $+1$.
Then the weight of vertices in $U$ are all $2$,
the weight of vertices in $V_1, V_2$ are all $1,3$, respectively,
and the weight of vertices in $W_1$ are all $-1$ and in $W_2$ are $+1$.
Hence, this labeling defines a SEDF $f$ with $w(f)=(3m+2n-2)/2$.

Note that since the weight of some vertices in $W$ is $-1$, no vertices in $U\cup V$ can
have weight zero or $-1$ by (\ref{eq1}).
Now by Lemma \ref{oeo.eoo} and the facts that $f(W)=-1$, $f(u)=2$
for every $u\in U$, $f(v)=1$ for every $v\in V_1$ and $f(v)=3$ for every $v\in V_2$, it follows that $\gamma_s'(K_{m,n,p})=(3m+2n-2)/2$.
\end{proof}

\section{The SEDNs of $K_{1,n,p}$ and $K_{2,2,p}$}\label{SEC5}
The constructions given in Section \ref{SEC4} work if the sum of the desired negative edges at vertices in $U$ and at vertices in $V$ is not less than the desired negative edges at vertices in $W$.
In this section we calculate the signed edge domination numbers of $K_{1,n,p}$ and $K_{2,2,p}$
which are not covered by constructions given in Section \ref{SEC4}.

\begin{lemma}\label{K1,n,p}
{\rm
Let $n\geq 1$ and $p\geq n+1$.
\begin{enumerate}
\item If $n,p$ are odd, then $\gamma_s'(K_{1,n,p})=n+2$.
\item If $n,p$ are even, then $\gamma_s'(K_{1,n,p})=n+2$.
\item If $n$ is even and $p$ is odd, then $\gamma_s'(K_{1,n,p})=2n+1$.
\item If $n$ is odd and $p$ is even, then $\gamma_s'(K_{1,n,p})=2n+3$.
\end{enumerate}
}
\end{lemma}

\begin{proof}
Consider the graph $K_{1,n,p}$ whose partite sets are $U,V$ and $W$.

\noindent {\bf Case 1.}\quad $n$ and $p$ are odd. \\
Since $mn-m-n=n-1-n=-1$ when $m=1$, the construction given in Theorem \ref{p>=m+n+1.ooo}
does not work.
On the other hand,
$$m(n+p-2)/2+n(m+p-2)/2-(p-1)(m+n)/2-(m+n-2)/2=0$$
when $m=1$.
So we can label $(n+p-2)/2+n(1+p-2)/2$ edges between $U\cup V$ and $W$ such that
the vertex in $U$ is incident with $(n+p-2)/2$ negative edges, each vertex in $V$ is incident with
$(p-1)/2$ negative edges and all the vertices in $W$ are incident with $(n+1)/2$
negative edges except one vertex which is incident with $(n-1)/2$ negative edges.
We label the remaining edges with $+1$. This yields a signed edge dominating function of weight
$n+2$. It is easy to see that this is the minimum weight of a SEDF of $K_{1,n,p}$ when $n,p$ are odd.
\vspace{3mm}

\noindent {\bf Case 2.}\quad $n$ and $p$ are even. \\
Since $mn-n-(3m+1)/2=-2$ when $m=1$, the construction given in Theorem \ref{p>=m+n+1.oee}
does not work. On the other hand,
$$\begin{array}{l}
m(n+p-2)/2+((n-m+1)/2)(m+p-1)/2\\+((n+m-1)/2)(m+p-3)/2-(p/2)(m+n+1)/2\\
-((p-2)/2)(m+n-1)/2-(m+n-3)/2=mn-n-(3m-3)/2=0
\end{array}$$
when $m=1$.
So we can label $m(n+p-2)/2+((n-m+1)/2)(m+p-1)/2\\+((n+m-1)/2)(m+p-3)/2$ edges between $U\cup V$ and $W$ such that
the vertex in $U$ is incident with $(n+p-2)/2$ negative edges, half of the vertices in $V$ are incident with
$p/2$ negative edges and the other half are incident with
$(p-2)/2$ negative edges, $p/2$ vertices in $W$ are incident with $(n+2)/2$, $(p-2)/2$ vertices are incident with $n/2$ and one vertex is incident with $(n-2)/2$ negative edges.
We label the remaining edges with $+1$. This yields a signed edge dominating function of weight
$n+2$. It is easy to see that this is the minimum weight of a SEDF of $K_{1,n,p}$ when $n,p$ are even.
\vspace{3mm}

\noindent {\bf Case 3.}\quad $n$ is even and $p$ is odd.\\
If $m=1$, then $mn-m-3n/2<0$, so the constructions given in Theorem \ref{p>=m+n+1.oeo}
do not work. On the other hand,
$$\begin{array}{l}
m(n+p-1)/2+((n-m+1)/2)(m+p-2)/2\\+((n+m-1)/2)(m+p-4)/2-((p-n-1)/2)(m+n+1)/2 \\-((p+n+1)/2)(m+n-1)/2=mn-m-n+1=0
\end{array}$$
when $m=1$.
Place $m(n+p-1)/2+((n-m+1)/2)(m+p-2)/2+((n+m-1)/2)(m+p-4)/2$  negative edges
between $U\cup V$ and $W$ such that the vertex in $U$ is incident with $(n+p-1)/2$ negative edges,
$(n-m+1)/2$ vertices in $V$ are incident with $(m+p-2)/2$ negative edges, $(n+m-1)$/2 vertices are incident
with $(m+p-4)/2$ negative edges, $(p-n-1)/2$ vertices of $W$ are incident with (m+n+1)/2 negative edges and
$(p+n+1)/2$ vertices of $W$ are incident with (m+n-1)/2 negative edges.
Label the remaining edges with $+1$. This yields a signed edge dominating function of weight $2n+m$ which is
$2n+1$ when $m=1$.
It is an easy to see that $\gamma(K_{1,n,p})=2n+1$.
\vspace{3mm}

\noindent {\bf Case 4.}\quad $n$ is odd and $p$ is even.\\
If $m=1$, then $mn-(3m+3n)/2<0$, so the constructions given in Theorem \ref{p>m+n.ooe}
do not work. On the other hand,
$$\begin{array}{l}
m(n+p - 3)/2 + n (m + p - 3)/2 -(p-(n+3)/2) (m + n)/2\\ -((n+3)/2)(m+n-2)/2=(2mn-3m-2n+3)/2=0
\end{array}$$
when $m=1$.
So we can find a signed edge dominating function of weight
$(1/2)[3+3n+2(n+3)/2]=2n+3$. It is easy to verify that $\gamma_s'(K_{1,n,p})=2n+3$.
\end{proof}

In \cite{Xu1} it was conjectured that $\gamma_s'(G)\leq |V(G)|-1$ for every graph $G$ of order at least 2.
Note that if $n$ is odd, then $\gamma_s'(K_{1,n,n+3})=2n+3$ by Lemma \ref{K1,n,p}, Part 4.
Hence, the graph $K_{1,n,n+3}$ achieves the upper bound in this conjecture.

\begin{lemma}\label{K2,2,p}
{\rm
Let $p\geq 5$ be odd. Then
$\gamma_s'(K_{2,2,p})= 8$.
}
\end{lemma}

\begin{proof}
Consider the graph $K_{m,n,p}$ with partite sets $U$, $V$ and $W$.
Partition $W$ into $W_1, W_2$ and $W_3$ such that $|W_1\mid=|W_2\mid$ and
$|W_3|=1$. Label the edges between $U$ and $W_1$ and between $V$ and $W_2$ with $-1$
and the remaining edges with $+1$. Then the weight of vertices in $W_1\cup W_2$ are zero and
the weight of the vertex in $W_3$ is $4$. The weight of the vertices in $U\cup V$ are all 3. This leads to $\gamma_s'(K_{2,2,p})=8$.
\end{proof}

\section{Main Theorem}\label{SEC6}
Let $m,n,p$ be positive integers, $m\leq n$ and $p\geq m+n$.
In this section we state the main theorem of this paper.
This result together with the main result of \cite{KG} provide
the signed edge domination number of $K_{m,n,p}$ for all positive integers $m,n$ and $p$.

\vspace{5mm}

\noindent {\bf Main Theorem}\quad
Let $m,n$ and $p$ be positive integers, $m\leq n$ and
$p \geq m+n$. Let $m\geq 2$ and if $p$ is odd, $(m,n)\neq (2,2)$.
\begin{itemize}
\item[{\bf A.}] If $m,n$ and $p$ are even, 
then $\gamma_s'(K_{m,n,p})=m+n$.

\item[{\bf B.}]
If $m, n$ and $p$ are odd and $m,n\geq 3$, 
then $\gamma_s'(K_{m,n,p})=m+n+1$.

\item[{\bf C.}]
If $m, n\geq 3$ are odd and $p$ is even, 
then
\begin{enumerate}
\item $\gamma_s'(K_{m,n,p})=\dfrac{3m+3n+2}{2}$ if $m+n\equiv 0 \pmod 4$,

\item $\gamma_s'(K_{m,n,p})=\dfrac{3m+3n}{2}$ if $m+n\equiv 2 \pmod 4$.
\end{enumerate}

\item[{\bf D.}]
If $m, n$ are even, $p$ is odd and $(m,n)\neq (2,2)$, 
then
\begin{enumerate}
\item $\gamma_s'(K_{m,n,p})=\dfrac{3m+3n}{2}$ if $m+n\equiv 0 \pmod 4$,

\item $\gamma_s'(K_{m,n,p})=\dfrac{3m+3n+2}{2}$ if $m+n\equiv 2\pmod 4$.
\end{enumerate}

\item[{\bf E.}]
If $m$ is odd, $n,p$ are even and $3\leq m<n$, 
then
\begin{enumerate}
\item $\gamma_s'(K_{m,n,p}) =\dfrac{3m+2n+1}{2}$ if $m\equiv 1\pmod 4$,

\item $\gamma_s'(K_{m,n,p}) =\dfrac{3m+2n-1}{2}$ if $m\equiv 3\pmod 4$.
\end{enumerate}

\item[{\bf F.}]
If $m,p$ are even, $n$ is odd and $m<n$, 
then
\begin{enumerate}
\item $\gamma_s'(K_{m,n,p})=\dfrac{2m+3n+1}{2}$ if $n\equiv 1\pmod 4$,
\item $\gamma_s'(K_{m,n,p})=\dfrac{2m+3n-1}{2}$ if $n\equiv 3\pmod 4$.
\end{enumerate}

\item[{\bf G.}]
If $m,p$ are odd, $n$ is even and $3\leq m<n$, 
then
\begin{enumerate}
\item $\gamma_s'(K_{m,n,p})=\dfrac{2m+3n}{2}$ if $n\equiv 0 \pmod 4$,

\item $\gamma_s'(K_{m,n,p})=\dfrac{2m+3n-2}{2}$ if $n\equiv 2 \pmod 4$.
\end{enumerate}

\item[{\bf H.}] If $m$ is even, $n,p$ are odd and $m<n$, 
then
\begin{enumerate}
\item $\gamma_s'(K_{m,n,p})=\dfrac{3m+2n}{2}$ if $m\equiv 0 \pmod 4$,

\item $\gamma_s'(K_{m,n,p})=\dfrac{3m+2n-2}{2}$ if $m\equiv 2 \pmod 4$.
\end{enumerate}
\end{itemize}

In addition, $\gamma_s'(K_{1,n,p})=n+2$ if $n,p$ are both odd or both even,
$\gamma_s'(K_{1,n,p})=2n+1$ if $n$ is even and $p$ is odd,
$\gamma_s'(K_{1,n,p})=2n+3$ if $n$ is odd and $p$ is even, and $\gamma_s'(K_{2,2,p})=8$ if $p$ is odd.


\end{document}